\documentclass[leqno,12pt]{article} 
\usepackage{amsmath, amssymb}
\usepackage{amsthm} 
\theoremstyle{plain} 
\newtheorem{theorem}{\indent\sc Theorem}[section] 
\newtheorem{lemma}[theorem]{\indent\sc Lemma}
\newtheorem{corollary}[theorem]{\indent\sc Corollary}

\theoremstyle{definition} 

\newcommand{\tententen}{
\begin{picture}(0,0)
\setlength{\unitlength}{1.08pt}
{\thicklines
\put(3,0){\circle*{1}}
\put(6,0){\circle*{1}}
\put(9,0){\circle*{1}}
}
\qbezier(12,0)(17,0)(22.3,0)
\end{picture}}
\newcommand{\koreraha}{
\begin{picture}(0,0)
\setlength{\unitlength}{1.08pt}
\qbezier(-22.5,10)(-8.25,15)(6,15)\qbezier(6,15)(20.25,15)(34.5,10)
\end{picture}}
\newcommand{\migite}{
\begin{picture}(0,0)
\setlength{\unitlength}{1.08pt}
{\thicklines
\put(6.05,0){\circle{12}}
}
\qbezier(12,0)(17,0)(22.3,0)
\end{picture}}
\newcommand{\hidariashi}{
\begin{picture}(0,0)
\setlength{\unitlength}{1.08pt}
{\thicklines
\put(6.05,0){\circle{12}}
}
\end{picture}}
\makeatletter
\long\def\@makecaption#1#2{%
\vskip\abovecaptionskip
\sbox\@tempboxa{#1. #2}
\ifdim \wd\@tempboxa >\hsize%
#1 #2\par
\else
\global \@minipagefalse
\hb@xt@\hsize{\hfil\box\@tempboxa\hfil}%
\fi
\vskip\belowcaptionskip}
\makeatother
\makeatletter
\def\address#1#2{\begingroup
\noindent\parbox[t]{7.8cm}{%
\small{\scshape\ignorespaces#1}\par\vskip1ex
\noindent\small{\itshape E-mail address}%
\/: #2\par\vskip4ex}\hfill%
\endgroup}%
\makeatother
\title{\uppercase{Extremal hyperelliptic fibrations on rational surfaces}} 
\author{
\textsc{Shinya Kitagawa} 
}
\date{} 
\begin{document}
\setlength{\baselineskip}{6.5mm}

\maketitle

\footnote{ 
2000 \textit{Mathematics Subject Classification}.
Primary 14J26; Secondary 14D06.
}
\footnote{ 
\textit{Key words and phrases}. 
Rational surfaces, fibrations, Mordell-Weil group.
}

\begin{abstract}
We consider a rational surface with a relatively minimal fibration. 
Picard number of a such fibred surface is bounded in terms of the genus of a general fibre. 
When Picard number is the maximum for any given genus, we characterize a such fibred surface 
whose Mordell-Weil group is trivial by singular fibres. 
Furthermore, we describe the defining equation explicitly. 
\end{abstract}

%
%
\section{Introduction}
The theory of the Mordell-Weil lattices are sufficiently developed by 
Oguiso and Shioda in \cite{OS} for minimal elliptic rational surfaces. 
In their work, the even unimodular root lattice $E_8$ of rank eight played very important role 
as the predominant frame. 
For example, it was shown that the Mordell-Weil group is trivial if and only if 
there exists a singular fibre of type $\mathrm{II}^*$ in the sense of Kodaira \cite{Kodai63} 
whose dual graph contains $E_8$ as a subgraph. 
The lattice $E_8$ also appears in another application by Shioda \cite{Shioda91} to describe 
a hierarchy of deformations of rational double points. 

Let $X$ be a smooth projective rational surface and 
$f:X\to \mathbb{P}^1$ a relatively minimal fibration 
whose general fibre is a projective curve of genus $g\geq2$. 
Although the Picard number $\rho(X)$ is letter than or equal to $4g+6$ in general, 
we consider the case $\rho(X)=4g+6$. 
Then the maximal Mordell-Weil lattice is isomorphic to a unimodular lattice called 
$D_{4g+4}^+$ in \cite[\S7, Ch.~4]{CS} of rank $4g+4$ (cf.~\cite{SS} and \cite{Viet00}). 
Furthermore, Saito \cite{Saito94} gives an example of $f:X\to\mathbb{P}^1$ whose Mordell-Weil group is trivial 
and which has an extension of a singular fibre of type $\mathrm{II}^*$. 
Since $D_8^+=E_8$, we expect an application as like the elliptic case \cite{Shioda91}. 

The goal of this paper is to give necessary and sufficient conditions for Mordell-Weil group of $f$ to be trivial 
by singular fibres of $f$, defining equations, and so on. 
The condition $(2\mathrm{a})$ or $(2\mathrm{b})$ in Theorem~\ref{MT} 
characterizes a reducible fibre of $f$ by the dual graph, though $f$ also have irreducible singular fibres. 
In particular, the kind of the reducible fibre of $f$ is unique. 
Here any irreducible fibre of $f$ does not affect the structure of Mordell-Weil group and lattice of $f$ in general. 
The condition $(4\mathrm{a})$ or $(5\mathrm{a})$ in Theorem~\ref{MT} describes 
how to construct $f:X\to\mathbb{P}^1$ whose Mordell-Weil group is trivial. 
Furthermore, in the conditions $(4\mathrm{b})$ and $(5\mathrm{b})$ corresponding to $(4\mathrm{a})$ and $(5\mathrm{a})$ 
respectively, we describe a defining equation of $f:X\to\mathbb{P}^1$.

In the case where $g=2$ and $\rho(X)=14$, from Theorem~~\ref{Thm:PM} and $(3)$ in Theorem~\ref{MT}, 
we have that there exist no birational morphisms $X\to\mathbb{P}^2$ if and only if Mordell-Weil group of $f$ is trivial. 
Furthermore, for any fibred rational surface of genus two whose Picard number is less than fourteen, 
we show that there exists a birational morphism from the surface to $\mathbb{P}^2$ in \cite{Kitagawa4}. 

\medskip

A{\sc cknowledgement}. 
The author would like to express his heartfelt gratitude to 
Professor Kazuhiro Konno for his valuable advice, guidance and encouragement. 

%
%
\section{Preliminaries}
We briefly review basic notation and results of fibred rational surfaces and Mordell-Weil lattices. 
Here, a fibred rational surface means a smooth projective rational surface $X/\mathbb{C}$ together with 
a relatively minimal fibration $f:X\to \mathbb{P}^1$ 
whose general fibre $F$ is a smooth projective curve of genus $g\geq1$. 
In particular, any fibre of $f$ is connected and contains no $(-1)$-curves as components. 
Since $X$ is rational, the first Betti number equals zero. 
The second Betti number is equal to Picard number $\rho(X)$ since the geometric genus is zero. 
Hence, we see that $\rho(X)=10-K_X^2=4g+6-(K_X+F)^2$ from Noether's formula. 
Furthermore, from \cite{slope}, we have that the adjoint divisor $(K_X+F)$ is nef when $g\geq2$. 
In particular, $\rho(X)\leq4g+6$. 

Via $f$, we can regard $X$ as a smooth projective curve of genus $g$ 
defined over the rational function field $\mathbb{K}=f^*\mathbb{C}(\mathbb{P}^1)$. 
We assume that it has a $\mathbb{K}$-rational point $O$. 
Let $\mathcal{J}_\mathcal{F}/\mathbb{K}$ be the Jacobian variety of the generic fibre $\mathcal{F}/\mathbb{K}$ of $f$. 
The Mordell-Weil group of $f$ is the group of $\mathbb{K}$-rational points $\mathcal{J}_\mathcal{F}(\mathbb{K})$. 
It is a finitely generated abelian group, since $X/\mathbb{C}$ is a rational surface. 
The rank of the group is called the {\it Mordell-Weil rank}. 
From \cite{Shioda90} and \cite{Shioda99}, it is given by 
\begin{align}
\mathrm{rk}\mathcal{J}_\mathcal{F}(\mathbb{K})
=\rho(X)-2-\sum_{t\in\mathbb{P}^1}(v_t-1), 
\label{Eq:MWrk}
\end{align}
where $v_t$ denotes the number of irreducible components of the fibre $f^{-1}(t)$. 
There is a natural one-to-one correspondence between 
the set of $\mathbb{K}$-rational points $\mathcal{F}(\mathbb{K})$ and the set of sections of $f$. 
For $P\in\mathcal{F}(\mathbb{K})$, we denote by $(P)$ the section 
corresponding to $P$ which is regarded as a horizontal curve on $X$. 
In particular, $(O)$ corresponding to the origin $O$ of $\mathcal{J}_\mathcal{F}(\mathbb{K})$ 
is called the {\it zero section}. 
Shioda's main idea in \cite{Shioda90} and \cite{Shioda99} 
is to view the free part of $\mathcal{J}_\mathcal{F}(\mathbb{K})$ 
as a Euclidean lattice with respect to a natural pairing induced by the intersection form on $H^2(X)$. 
The lattice is called the {\it Mordell-Weil lattice} of $f$. 
We denote by $\mathrm{MWL}(f)$ it. 
In fact, by describing the N\'eron-Severi group $\mathrm{NS}(X)$, 
we can explicitly determine the structure of $\mathrm{MWL}(f)$ as follows: 
Let $T$ be the subgroup of $\mathrm{NS}(X)$ generated by $(O)$ and all the irreducible components of fibres of $f$. 
When we equip $\mathrm{NS}(X)$ and $T$ with the bilinear form which is $(-1)$ times of the intersection form, 
we call them the N\'eron-Severi lattice $\mathrm{NS}(X)^-$ and the trivial lattice $T^-$ respectively. 
Since $X$ is a rational surface, $\mathrm{NS}(X)^-$ is a unimodular lattice, that is, 
the absolute value of the determinant of the Gram matrix equals one. 
Then the following holds. 
\begin{theorem}[{\upshape \cite{Shioda90}, \cite{Shioda99}}]\label{Thm:MW}
Keep the notation and assumptions. 
Put $\hat{T}=(T\otimes\mathbb{Q})\cap\mathrm{NS}(X)$. 
Then 
\begin{align*}
\mathcal{J}_\mathcal{F}(\mathbb{K})\simeq\mathrm{NS}(X)/T, \ \ \ \mathcal{J}_\mathcal{F}(\mathbb{K})_{\mathrm{tor}}\simeq\hat{T}/T.
\end{align*}
Let $L$ be the orthogonal complement $(T^-)^\perp\subset\mathrm{NS}(X)^-$. 
Then the dual lattice 
\begin{align*}
L^*=\{ \mathfrak{x}\in L\otimes\mathbb{Q} \ | \ \langle \mathfrak{x}, \mathfrak{y} \rangle 
	_{L\otimes\mathbb{Q}}\in\mathbb{Z}, \ \ ^\forall\mathfrak{y}\in L\}
\end{align*}
is isomorphic to $\mathrm{MWL}(f)$. 
\end{theorem}

Now, for a non-negative integer $d$, we put 
\begin{align*}
\Sigma_d = \{ ((X_0:X_1:X_2), (Y_0:Y_1))|X_1Y_1^d=X_2Y_0^d \} \subset \mathbb{P}^2\times\mathbb{P}^1
\end{align*}
and call it Hirzebruch surface of degree $d$. 
The restriction of the second projection to $\Sigma_d$ gives a structure of $\mathbb{P}^1$-bundle. 
We also remark that $\Sigma_0\simeq\mathbb{P}^1\times\mathbb{P}^1$. 
Conversely, any $\mathbb{P}^1$-bundle over $\mathbb{P}^1$ is isomorphic to $\Sigma_d$ for some $d$. 
We often consider on Zariski open defined by $X_0Y_0\not=0$ 
and take $(x,y)=(X_1/X_0, Y_1/Y_0)$ as an affine coordinate. 
Let $\Delta_{[d]}$ be a minimal section of $\Sigma_d$ defined by $x=0$ and 
$\Gamma_{[d]}$ the fibre defined by $y=0$. 

To clarify the structure of the Mordell-Weil lattice, 
we choose a ruling on $X$ and its relatively minimal model $\Sigma_d$ 
carefully so that we get a natural $\mathbb{Z}$-basis of $\mathrm{NS}(X)$ 
which gives us a simple presentation of $F$. 
This is done by choosing a birational morphism $\mu:X\to\Sigma_d$ 
which contracts step by step a $(-1)$-curve whose intersection number with $F$ 
is the smallest among all $(-1)$-curves. 
When $d=0$, we may assume that $\Gamma_{[0]}.\mu_*F<\Delta_{[0]}.\mu_*F$ without loss of generality. 
We call a $(-1)$-curve on $X$ a $(-1)$-section of $f$ if the intersection number with $F$ is equal to one. 
Then the following holds. 

\begin{theorem}[cf.~\cite{SS} and \cite{Viet00}]\label{Thm:PM}
Let $X$ be a smooth rational surface and $f:X\to\mathbb{P}^1$ a relatively minimal fibration of genus $g\geq1$. 
Assume that $\rho(X)=4g+6$ and $f$ has no multiple fibre when $g=1$. 
Then there exists a birational morphism $\mu:X\to\Sigma_d$ with $d\leq g+1$ such that 
the following conditions $(\mathrm{i})$, $(\mathrm{ii})$ hold. 
\begin{itemize}
\item[$(\mathrm{i})$] $\mu_*F$ is linearly equivalent to $(2\Delta_{[d]}+(g+d+1)\Gamma_{[d]})$. 
\item[$(\mathrm{ii})$] The pull-back to $X$ of a $(-1)$-curve contracted by $\mu$ intersects with $F$ at just one point. 
\end{itemize}
In particular, $F$ is a hyperelliptic curve and $f$ has at least one $(-1)$-section. 
\end{theorem}
\begin{proof}
When $g=1$, the assertion is well-known. 
In particular, we have that $f$ is anti-canonical map $\Phi_{|-K_X|}$ 
from the canonical bundle formula and the assumption, 
where $|D|$ means the complete linear system of $D$. 

Assume that $g\geq2$. 
We can find at least one base-point-free pencil of rational curves on $X$. 
We choose among them a pencil $|R|$ of rational curves with $R^2=0$ in such a way that $(K_X+F).R$ is minimal. 
Here, $(K_X+F).R$ is non-negative since  $(K_X+F)$ is nef. 
If $(K_X+F).R>0$, then we have $(K_X+F)^2>0$ from \cite{KK} and \cite{Kitagawa3}. 
It contradicts to $\rho(X)=4g+6$. 
Thus we have $(K_X+F).R=0$. 
We take a relatively minimal model of $X$ with respect to $\Phi_{|R|}$ and consider the image of $F$. 
Then we perform a succession of elementary transformations (\cite{Hartshorne}) at singular points of the image curve 
to arrive at a particular relatively minimal model called a $\#$-minimal model in \cite{Iitaka}. 
Let $\mu:(X, F)\to(\Sigma_d, \mu_*F)$ be a birational morphism to a $\#$-minimal model of $(X,F)$. 
Then we have that $\mu$ satisfies the conditions $(\mathrm{i})$, $(\mathrm{ii})$ 
from $F.R=-K_X.R=2$ and properties of $\#$-minimal model. 
\end{proof}

When $g=1$ and $f$ has no multiple fibre, 
we have $F.R=-K_X.R=2$ for all general fibres $R$ of rulings on $X$. 
Therefore, $f\times\Phi_{|R|}:X\to\varSigma_0:=\mathbb{P}^1\times\mathbb{P}^1$ is a generically finite double cover 
and the branch divisor is linearly equivalent to $(4\varDelta_{[0]}+2\varGamma_{[0]})$. 

We assume that $g\geq2$ and $\rho(X)=4g+6$. 
Any fibred rational surface $f:X\to\mathbb{P}^1$ can be considered as a subpencil 
$\Lambda\subset|2\Delta_{[d]}+(d+g+1)\Gamma_{[d]}|$ through 
a birational morphism as in Theorem~\ref{Thm:PM}. 
Conversely, any fibred rational surface is obtained from a subpencil 
$\Lambda\subset|2\Delta_{[d]}+(d+g+1)\Gamma_{[d]}|$ by blowing $\Sigma_d$ up at the $(4g+4)$ base points. 
Let $(O)$ be a $(-1)$-section of $f$ and $U^-$ a sublattice generated by $(O)$ and $F$ in $\mathrm{NS}(X)^-$. 
Then the orthogonal complement $(U^-)^\perp\subset\mathrm{NS}(X)^-$ is isomorphic to 
a unimodular lattice called $D_{4g+4}^+$ in \cite{CS} of rank $4g+4$, 
which contains $D_{4g+4}$ as the maximal root sublattice at index two. 
In fact, by choosing a $\mathbb{Z}$-basis of $(U^-)^\perp$ as in \cite{SS} and \cite{Viet00}, 
we have an extended Dynkin diagram in Figure~$\ref{fig:D_4g+4^+}$ below. 
Furthermore, all fibres of $f$ are irreducible if and only if $T=U$, and then we have 
$\mathrm{MWL}(f)\simeq D_{4g+4}^+$ from Theorem~\ref{Thm:MW}. 
We can really construct a such $f:X\to\mathbb{P}^1$. 
This case was studied in \cite{SS}. 

We take a birational morphism $\mu:X\to\Sigma_d$ as in Theorem~\ref{Thm:PM}. 
We have that $\mu^*\Gamma_{[d]}$ is linearly equivalent to $(K_X+F)/(g-1)$ and 
consider the ruling $\Phi_{|(K_X+F)/(g-1)|}:X\to\mathbb{P}^1$. 
All $(-1)$-sections of $f$ and $(-2)$-curves contained in fibres of $f$ are not intersect with $(K_X+F)$. 
Hence they are components of degenerate fibres of the ruling $\Phi_{|(K_X+F)/(g-1)|}$. 
Conversely, irreducible components of degenerate fibres of $\Phi_{|(K_X+F)/(g-1)|}$ is obtained from 
base points of a subpencil $\Lambda\subset|2\Delta_{[d]}+(d+g+1)\Gamma_{[d]}|$ by blowing $\Sigma_d$ up at them. 
Therefore, we have the following.

\begin{corollary}\label{Cor:DFofRuling}
Assume that $g\geq2$ and $\rho(X)=4g+6$. 
Then any degenerate fibre of the ruling $\Phi_{|(K_X+F)/(g-1)|}:X\to\mathbb{P}^1$ consists of 
$k$ $(-2)$-curves contained in a fibre of $f$ and one or two $(-1)$-sections of $f$, 
and the dual graph is Dynkin diagram of a root lattice $A_{k+2}$ as in {\upshape Figure}~$\ref{fig:DFofRulingA}$ 
or $D_{k+1}$ as in {\upshape Figure}~$\ref{fig:DFofRulingD}$. 
\begin{figure}[hbtp]
\begin{center}
\begin{picture}(235,58)
\setlength{\unitlength}{1.08pt}
\put(15,25){\makebox(0,0){Type $A_{k+2}:$}}
\put(80.5,50){\makebox(0,0){Case $k=0$}}
\put(71.3,10){\makebox(0,0){$1$}}
\put(93.8,10){\makebox(0,0){$1$}}
\put(65,10){\migite}
\put(87.5,10){\hidariashi}
%
%
\put(71,-2){\makebox(0,0){$\scriptstyle 1$}}
\put(93.5,-2){\makebox(0,0){$\scriptstyle 1$}}

\put(181.5,50){\makebox(0,0){Case $k\geq1$}}
\put(138.8,10){\makebox(0,0){$1$}}
\put(228.8,10){\makebox(0,0){$1$}}
\put(132.5,10){\migite}
\put(155,10){\migite}
\put(177.5,10){\tententen}
\put(177.5,10){\koreraha}
\put(200,10){\migite}
\put(222.5,10){\hidariashi}
%
%
\put(138.5,-2){\makebox(0,0){$\scriptstyle 1$}}
\put(161,-2){\makebox(0,0){$\scriptstyle 1$}}
\put(183.5,33){\makebox(0,0){$\scriptstyle k$}}
\put(206,-2){\makebox(0,0){$\scriptstyle 1$}}
\put(228.5,-2){\makebox(0,0){$\scriptstyle 1$}}
\end{picture}
\caption{}\label{fig:DFofRulingA}
\end{center}
\end{figure}
\begin{figure}[hbtp]
\begin{center}
\begin{picture}(270,58)
\setlength{\unitlength}{1.08pt}
\put(15,25){\makebox(0,0){Type $D_{k+1}:$}}
\put(81.5,50){\makebox(0,0){Case $k=2$}}
\put(83.8,10){\makebox(0,0){$1$}}
\put(55,10){\migite}\put(77.5,10){\migite}\put(100,10){\hidariashi}
%
%
\put(61,-2){\makebox(0,0){$\scriptstyle 1$}}
\put(83.5,-2){\makebox(0,0){$\scriptstyle 2$}}
\put(106,-2){\makebox(0,0){$\scriptstyle 1$}}

\put(205.25,50){\makebox(0,0){Case $k\geq3$}}
\put(151.3,10){\makebox(0,0){$1$}}
{\thicklines
\put(241.3,32.5){\circle{12}}
}
\qbezier(241.3,16.5)(241.3,21.5)(241.3,26.3)
\put(145,10){\migite}
\put(167.5,10){\migite}
\put(190,10){\tententen}
\put(190,10){\koreraha}
\put(212.5,10){\migite}
\put(235,10){\migite}
\put(257.5,10){\hidariashi}
%
\put(252.5,32.8){\makebox(0,0){$\scriptstyle 1$}}
\put(151,-2){\makebox(0,0){$\scriptstyle 2$}}
\put(173.5,-2){\makebox(0,0){$\scriptstyle 2$}}
\put(196,33){\makebox(0,0){$\scriptstyle k-3$}}
\put(218.5,-2){\makebox(0,0){$\scriptstyle 2$}}
\put(241,-2){\makebox(0,0){$\scriptstyle 2$}}
\put(263.5,-2){\makebox(0,0){$\scriptstyle 1$}}
\end{picture}
\caption{}\label{fig:DFofRulingD}
\end{center}
\end{figure}
Conversely, $(-1)$-sections of $f$ and $(-2)$-curves contained in fibres of $f$ are components of 
degenerate fibres of the ruling $\Phi_{|(K_X+F)/(g-1)|}:X\to\mathbb{P}^1$. 
\end{corollary}

The hyperelliptic involution of $f:X\to\mathbb{P}^1$ naturally induces a double cover as follows. 

\begin{corollary}\label{Cor:FDC}
If $g\geq2$ and $\rho(X)=4g+6$, then $f\times \Phi_{|(K_X+F)/(g-1)|}:X\to\varSigma_0$ is 
a generically finite double cover and the branch divisor is linearly equivalent to 
$((2g+2)\varDelta_{[0]}+2\varGamma_{[0]})$. 
Conversely, the finite double cover of $\varSigma_0$ branched along a reduced curve which is linearly equivalent to 
$((2g+2)\varDelta_{[0]}+2\varGamma_{[0]})$ with a minimal resolution of the singularity gives 
a hyperelliptic fibration of genus $g$ on a smooth rational surface whose Picard number is $(4g+6)$. 
\end{corollary}

%
%
\section{Main theorem}
Let $X$ be a smooth projective rational surface with $\rho(X)=4g+6$ and 
$f:X\to\mathbb{P}^1$ a relatively minimal fibration of genus $g\geq1$. 
When $g=1$, we assume that $f$ has no multiple fibre, or $f$ has a section. 
Furthermore, we consider an extremal case as follows.

\begin{theorem}\label{MT}
Let $f:X\to\mathbb{P}^1$ be as above and $\mathbb{K}=f^*(\mathbb{C}(\mathbb{P}^1))$. 
Then the following assertions are equivalent. 
\begin{itemize}
\item[$(1)$]
Mordell-Weil group of $f$ is trivial.
\item[$(2\mathrm{a})$]
$f$ has a reducible fibre whose dual graph is as in {\upshape Figure~$\ref{fig:VIII-4}$}. 
\begin{figure}[hbtp]
\begin{center}
\begin{picture}(189.5,38)
\setlength{\unitlength}{1.08pt}
{\thicklines
\put(118.8,32.5){\circle{12}}
\qbezier(157.5,10)(157.5,16.5)(173.5,16.5)
\qbezier(157.5,10)(157.5,3.5)(173.5,3.5)
\qbezier(189.5,10)(189.5,16.5)(173.5,16.5)
\qbezier(189.5,10)(189.5,3.5)(173.5,3.5)
}
\qbezier(118.8,16.5)(118.8,21.5)(118.8,26.3)
\put(0,10){\migite}
\put(22.5,10){\migite}
\put(45,10){\tententen}
\put(45,10){\koreraha}
\put(67.5,10){\migite}
\put(90,10){\migite}
\put(112.5,10){\migite}
\put(135,10){\migite}
\put(173.5,10){\makebox(0,0){\footnotesize $g+1$}}
\put(138,32.8){\makebox(0,0){$\scriptstyle 2g+1$}}
\put(6,-2){\makebox(0,0){$\scriptstyle 1$}}
\put(28.5,-2){\makebox(0,0){$\scriptstyle 2$}}
\put(51,33){\makebox(0,0){$\scriptstyle 4g-1$}}
\put(73.5,-2){\makebox(0,0){$\scriptstyle 4g$}}
\put(96,-2){\makebox(0,0){$\scriptstyle 4g+1$}}
\put(118.5,-2){\makebox(0,0){$\scriptstyle 4g+2$}}
\put(141,-2){\makebox(0,0){$\scriptstyle 2g+2$}}
\put(173.5,23){\makebox(0,0){$\scriptstyle 2$}}
\end{picture}
\caption{}\label{fig:VIII-4}
\end{center}
\end{figure}
\item[$(2\mathrm{b})$]
$f$ has a reducible fibre whose dual graph contains the extended Dynkin diagram of 
a unimodular integral lattice $D_{4g+4}^+$ as in {\upshape Figure~$\ref{fig:D_4g+4^+}$} as a subgraph. 

\begin{figure}[hbtp]
\begin{center}
\begin{picture}(189.5,33)
\setlength{\unitlength}{1.08pt}
{\thicklines
\put(118.8,27.5){\circle{12}}
\qbezier(157.5,5)(157.5,11.5)(173.5,11.5)
\qbezier(157.5,5)(157.5,-1.5)(173.5,-1.5)
\qbezier(189.5,5)(189.5,11.5)(173.5,11.5)
\qbezier(189.5,5)(189.5,-1.5)(173.5,-1.5)
}
\qbezier(118.8,11.5)(118.8,16.5)(118.8,21.3)
\put(0,5){\migite}
\put(22.5,5){\migite}
\put(45,5){\tententen}
\put(45,5){\koreraha}
\put(67.5,5){\migite}
\put(90,5){\migite}
\put(112.5,5){\migite}
\put(135,5){\migite}
\put(173.5,5){\makebox(0,0){\footnotesize $g+1$}}
%
\put(51,28){\makebox(0,0){$\scriptstyle 4g-2$}}
\end{picture}
\caption{}\label{fig:D_4g+4^+}
\end{center}
\end{figure}
\item[$(3)$]
$X$ has a unique ruling and possibilities of its relatively minimal models are 
Hirzebruch surfaces $\Sigma_g$ of degree $g$ and $\Sigma_{g+1}$ of degree $g+1$ only. 
In particular, there exists no birational morphisms $X\to\mathbb{P}^2$ if $g\geq2$. 
\item[$(4\mathrm{a})$]
$f:X\to\mathbb{P}^1$ is obtained from $\Sigma_g$ by eliminating the base points of the following pencil $\Lambda:$ 
Let $\Delta_{[g]}$ be the minimal section and $\Gamma_{[g], 0}$ a fibre of $\Sigma_g$. 
Take a curve $H_{[g]}$ which is linearly equivalent to $(2\Delta_{[g]}+(2g+1)\Gamma_{[g],0})$ 
and which is tangent to $\Gamma_{[g]}, 0$ at the intersection point of $\Gamma_{[g], 0}$ with $\Delta_{[g]}$. 
Then $(2\Delta_{[g]}+(2g+1)\Gamma_{[g], 0})$ and $H_{[g]}$ generate the pencil $\Lambda$. 
\item[$(4\mathrm{b})$]
There exist elements $t$, $x$, $y$ in $\mathbb{C}(X)$ and complex numbers $c_{i,j}$, $i=0,1,2$, $j=0, 1, \ldots, ig+1$ 
such that they satisfy the following$:$ 

$(\mathrm{I})$ \ $\mathbb{C}(X)=\mathbb{C}(x,y)$ and $\mathbb{K}=\mathbb{C}(t)$. 

$(\mathrm{II})$ \ $x^2y^{2g+1}=t\sum c_{i,j}x^iy^j$, \ \ \ $c_{1,0}=c_{0,0}=0$ and $c_{2,0}c_{0,1}\not=0$.

\item[$(5\mathrm{a})$]
$f:X\to\mathbb{P}^1$ is obtained from $\varSigma_0=\mathbb{P}^1\times\mathbb{P}^1$ as follows$:$ 
Let $\varDelta_{[0],0}$ be a section of $\varSigma_0$ with ${\varDelta_{[0],0}}^2=0$ and 
$\varGamma_{[0],0}$ a fibre of $\varSigma_0$. 
Take a curve $B_{[0],0}$ which is linearly equivalent to $((2g+1)\varDelta_{[0],0}+\varGamma_{[0],0})$ 
and which has a contact of order $(2g+1)$ with $\varGamma_{[0],0}$ 
at the intersection point of $\varGamma_{[0],0}$ with $\varDelta_{[0],0}$. 
Let $\pi^\flat:X^\flat\to\varSigma_0$ be the finite double cover branched along 
$(\varDelta_{[0],0}+\varGamma_{[0],0}+B_{[0],0})$. 
Then $X^\flat$ has only one rational double point of type $D_{4g+4}$ as its singularity. 
Take $\varsigma:X\to X^\flat$ be a minimal resolution of the singularity. 
Define $f:X\to\mathbb{P}^1$ as the composite of $\pi^\flat\circ\varsigma$ and 
the projection $\Phi_{|\varGamma_{[0],0}|}$ of $\varSigma_0$. 
\item[$(5\mathrm{b})$]
There exist $(2g+1)$ complex numbers $b_{1,1}, \ldots, b_{1,2g+1}$, two non-zero complex numbers $b_{0,2g+1}$, $b_{1,0}$ 
and $t$ in $\mathbb{K}$ such that the following holds$:$ 
$\mathbb{K}=\mathbb{C}(t)$ and $\mathbb{C}(X)$ is isomorphic to the quotient field of $\mathbb{C}[t,y,z]/(\psi(t,y,z))$, 
where

\hspace*{\fill}
$\psi(t,y,z)=z^2-ty(b_{0,2g+1}y^{2g+1}+b_{1,0}t+ty\sum_{j=1}^{2g+1}b_{1,j}y^{j-1})$. 
\hspace*{\fill}
\end{itemize}
\end{theorem}
In order to show Theorem~\ref{MT}, we prove some lemmas. 
\begin{lemma}\label{Lem:2b}
Assume that $f$ has a reducible fibre whose dual graph is as in {\upshape Figure}~$\ref{fig:VIII-4}$. 
Then $(\mathrm{I})$, $(\mathrm{II})$, $(\mathrm{III})$ and $(\mathrm{IV})$ hold$:$ 
\begin{itemize}
\item[$(\mathrm{I})$]
There exists a birational morphism $\mu:X\to\Sigma_g$ such that the pencil $\Lambda$ as in 
$(4\mathrm{a})$ in {\upshape Theorem~$\ref{MT}$} is obtained from a base-point-free pencil $|F|$ as images by $\mu$. 
\item[$(\mathrm{II})$]
Mordell-Weil group of $f$ is trivial. 
In particular, a $(-1)$-section of $f$ is unique. 
\item[$(\mathrm{III})$]
If an irreducible reduced curve $C$ on $X$ is neither the $(-1)$-section 
nor any component of the reducible fibre of $f$, then $C^2\geq0$. 
\item[$(\mathrm{IV})$]
$X$ has a unique ruling. 
Furthermore, $\Sigma_g$ and $\Sigma_{g+1}$ only can be its relatively minimal model. 
The both models are obtained from birational morphisms as in {\upshape Theorem~$\ref{Thm:PM}$}. 
\end{itemize}
\end{lemma}
\begin{proof}
$(\mathrm{I}):$ \ 
Let $\Theta_k$, $k=0,1, \cdots, 4g+4$ be components of the reducible fibre such that 
\begin{align*}
(\Theta_{i-1}.\Theta_{j-1})_{1 \leq i, j \leq 4g+5}=
\left(
\begin{array}{cccccccc}
-2 & 1 & 0 & \cdots & 0 & 0 & 0 & 0\\
1 & -2 & \ddots & \ddots & \vdots & \vdots & \vdots & \vdots\\
0 & \ddots & \ddots & 1 & 0 & \vdots & \vdots & \vdots\\
\vdots & \ddots & 1 & -2 & 1 & 0 & 0 & 0\\
0 & \cdots & 0 & 1 & -2 & 1 & 1 & 0\\
0 & \cdots & \cdots & 0 & 1 & -2 & 0 & 1\\
0 & \cdots & \cdots & 0 & 1 & 0 & -2 & 0\\
0 & \cdots & \cdots & 0 & 0 & 1 & 0 &-g-1
\end{array}
\right).
\end{align*}
We show that $f$ has a $(-1)$-section $E_{4g+4}$ from Theorem~$\ref{Thm:PM}$. 
Since $\Theta_0$ is a unique component whose multiplicity is one, $E_{4g+4}$ intersect with $\Theta_0$. 
We contract $E_{4g+4}, \Theta_0, \Theta_1, \ldots, \Theta_{4g+2}$ in turn. 
Since $\rho(X)=4g+6$, the image of $X$ by this birational morphism $\mu$ is $\Sigma_d$ for some $d$. 
Then $(\mu_*\Theta_{4g+3})^2=0$ and $(\mu_*\Theta_{4g+4})^2=-g$. 
Hence we have $d=g$ and $\mu_*\Theta_{4g+4}=\Delta_{[g]}$. 
Let $\Gamma_{[g], 0}$ be a fibre $\mu_*\Theta_{4g+3}$ of $\Sigma_g$. 
In fact, $\mu$ is a birational morphism as in Theorem~$\ref{Thm:PM}$. 
We take the image by $\mu$ of a general fibre of $f$ as $H_{[g]}$. 
Then the first assertion follows. 

$(\mathrm{II}):$ \ 
We consider $\mu$ as in the proof of $(\mathrm{I})$. 
Let $E_i$, $i=1,2,\ldots,4g+4$ be the pull-back to $X$ of $(4g+4)$ $(-1)$-curves contracted by 
the birational morphism $\mu:X\to\Sigma_g$. 
We may assume that $\Theta_k=E_{4g+3-k}-E_{4g+4-k}$, $k=0,1,\ldots,4g+2$. 
For simplicity, we denote the pull-back to $X$ of $\Delta_{[g]}$ and $\Gamma_{[g], 0}$ by the same symbols. 
We remark that $\Theta_{4g+3}=\Gamma_{[g], 0}-E_1-E_2$ and $\Theta_{4g+4}=\Delta_{[g]}-E_1$. 
Since 
$\mathrm{NS}(X)=\mathbb{Z}\Delta_{[g]}\oplus\mathbb{Z}\Gamma_{[g], 0}\oplus\bigoplus_{i=1}^{4g+4}\mathbb{Z}E_i$, 
we see that $E_{4g+4}$ and $\Theta_k$, $k=0,1,\ldots,4g+4$ also form $\mathbb{Z}$-basis of $\mathrm{NS}(X)$. 
Furthermore, we have that $\Theta_1, \Theta_2, \ldots, \Theta_{4g+4}$, $F$ and $E_{4g+4}$ also form 
$\mathbb{Z}$-basis of $\mathrm{NS}(X)$ by considering 
\begin{align*}
\Theta_0=F-\sum_{k=1}^{4g+1}(k+1)\Theta_k-(2g+2)\Theta_{4g+2}-(2g+1)\Theta_{4g+3}-2\Theta_{4g+4}. 
\end{align*}
Therefore, Mordell-Weil group of $f$ is trivial from Theorem~\ref{Thm:MW}. 

$(\mathrm{III}):$ \ 
Keep the notation as above. 
Let $C$ be an irreducible reduced curve on $X$. 
We put $C\sim\alpha\Delta_{[g]}+\beta\Gamma_{[g], 0}-\sum_{i=1}^{4g+4}m_iE_i$ 
for some integers $\alpha$, $\beta$ and $m_i$, $i=1,2,\ldots, 4g+4$. 
We assume that $C$ is neither $E_{4g+4}$ nor $\Theta_k$ for all of $k$. 
Since $C.E_{4g+4}$ and $C.\Theta_k$ are non-negative, 
we have $\alpha\geq m_1+m_2$, $\beta\geq \alpha g+m_1$ and 
$m_1\geq m_2\geq\cdots\geq m_{4g+4}\geq0$. 
Then 
\begin{align*}
C^2&\geq -\alpha ^2 g+2\alpha\beta -m_1^2-(4g+3)m_2^2\\
   &\geq  \alpha ^2 g+2\alpha m_1-m_1^2-(4g+3)m_2^2\\
   &\geq  (m_1+m_2)^2g+2(m_1+m_2)m_1-m_1^2-(4g+3)m_2^2\\
   &= (g+1)(m_1+3m_2)(m_1-m_2)\\
   &\geq 0. 
\end{align*}

$(\mathrm{IV}):$ \ 
We consider the birational morphism $\mu:X\to\Sigma_g$ as in the proof of $(\mathrm{I})$ and keep the notation. 
In particular, we remark that $\mu(E_1)$ is on $\Delta_{[g]}$. 
By performing an elementary transformation $\Sigma_g\dashrightarrow\Sigma_{g+1}$ at $\mu(E_1)$, 
we have another birational morphism $\mu':X\to\Sigma_{g+1}$ as the composite of $\mu$ and it. 
Here $\mu'$ contracts $\Theta_{4g+3}$ in place of $\Theta_{4g+2}$. 
In fact, $\mu$ and $\mu'$ give relatively minimal models of the same ruling defined by $|\sum_{k=0}^{4g+3}\Theta_k|$. 
From $(\mathrm{II})$ and $(\mathrm{III})$, there are no other birational morphisms $X\to\Sigma_d$. 
Thus the assertion $(\mathrm{IV})$ follows. 
\end{proof}
Next,  we show the converse of $(\mathrm{II})$ in {\upshape Lemma}~\ref{Lem:2b}: 
\begin{lemma}\label{Lem:1}
Assume that Mordell-Weil group of $f$ is trivial. 
Then $f$ has a unique reducible fibre and the dual graph is as in {\upshape Figure~$\ref{fig:VIII-4}$}. 
\end{lemma}
\begin{proof}
From Theorem~\ref{Thm:PM}, there exists a birational morphism $\mu:X\to\Sigma_d$ 
in order that $F$ is linearly equivalent to $2\Delta_{[d]}+(d+g+1)\Gamma_{[d]}-\sum_{i=1}^{4g+4}E_i$. 
We shall denote the $(-1)$-section of $f$ by $E_{4g+4}$. 
In particular, we remark that a section of $f$ is unique from the assumption. 
Therefore, in the process of contracting by $\mu$, 
we may assume that the point corresponding to $E_{i+1}$ is an infinitely near point of 
that to $E_i$ for $i=1,2,\ldots,4g+3$. 
Similarly, $\mu(E_2)$ corresponds a tangential direction at $\mu(E_1)$ of a fibre $\Gamma_{[d]}$ of $\Sigma_d$. 
Since $(-2)$-curves $\Gamma_{[d]}-E_1-E_2$ and $E_i-E_{i+1}$, $i=1,2,\ldots,4g+3$ are connected, 
a reducible singular fibre of $f$ contains all of them. 
However, they do not generate the reducible fibre of $f$, and we recall that 
$\Delta_{[d]}$, $\Gamma_{[d]}$ and $E_i$, $i=1, 2, \ldots, 4g+4$ form $\mathbb{Z}$-basis of $\mathrm{NS}(X)$. 
Hence, from $\rho(X)=4g+6$ and the equation of Mordell-Weil rank $(\ref{Eq:MWrk})$, we have that 
another component of the reducible fibre is unique and all other fibres of $f$ are irreducible. 
Furthermore, the assumption and Theorem~\ref{Thm:MW} imply that 
components $\Gamma_{[d]}-E_1-E_2$, $E_i-E_{i+1}$, $i=1, 2, \ldots, 4g+2$, 
the other component $\Theta$ and the unique section $E_{4g+4}$ form $\mathbb{Z}$-basis of $\mathrm{NS}(X)$. 
Remark that $\Theta.(\Gamma_{[d]}-E_1-E_2)$ and $\Theta.(E_i-E_{i+1})$ are non-negative. 
Then, $\Theta$ is $\Delta_{[d]}+\beta\Gamma_{[d]}$ or $\Delta_{[d]}+\beta\Gamma_{[d]}-E_1$ 
for some non-negative integer $\beta$ in a way similar to the proof of $(\mathrm{III})$ in Lemma~$\ref{Lem:2b}$. 
Here, $\Theta^2\geq0$ if $\beta>0$. 
It contradicts to the fact that $\Theta$ is a proper component of the reducible fibre of $f$. 
By considering $\Theta.F=0$, from $\Theta=\Delta_{[d]}$ and $\Theta=\Delta_{[d]}-E_1$, 
we have $d=g+1$ and $d=g$ respectively. 
Either way, $\Theta$, $\Gamma_{[d]}-E_1-E_2$ and $E_i-E_{i+1}$ 
form a singular fibre whose dual graph is as in {\upshape Figure~$\ref{fig:VIII-4}$}. 
\end{proof}

In order to show that $(3)$ implies $(2\mathrm{a})$ in Theorem~$\ref{MT}$, 
we use a weaker assumption than $(3)$ as follows：

\begin{lemma}\label{Lem:3}
Assume that $d=g$ or $g+1$ for all birational morphisms $X\to\Sigma_d$ satisfying conditions 
$(\mathrm{i})$, $(\mathrm{ii})$ in {\upshape Theorem~\ref{Thm:PM}}. 
Then $f$ has a reducible fibre whose dual graph is as in {\upshape Figure}~$\ref{fig:VIII-4}$. 
\end{lemma}
\begin{proof}
Let $\mu':X\to\Sigma_{g+1}$ be a birational morphism satisfying 
conditions $(\mathrm{i})$, $(\mathrm{ii})$ in Theorem~$\ref{Thm:PM}$. 
Then a base-point-free pencil $|F|$ can be considered as a subpencil 
$\Lambda'\subset|2(\Delta_{[g+1]}+(g+1)\Gamma_{[g+1]})|$ with $(4g+4)$ simple base points on $\Sigma_{g+1}$ through $\mu'$. 
Since $\Delta_{[g+1]}.2(\Delta_{[g+1]}+(g+1)\Gamma_{[g+1]})=0$, no base points of $\Lambda'$ lie on $\Delta_{[g+1]}$. 
We consider a base point $p_1'$ on $\Sigma_{g+1}$ and corresponding $(-1)$-curve $E_1'$. 
Let $\mu_1':X\to X_1$ be the birational morphism which contracts $(4g+3)$ $(-1)$-curves 
except $E_1'$ among $(4g+4)$ $(-1)$-curves contracted by $\mu'$. 
The composite of $\mu'$ and the elementary transformation $\Sigma_{g+1}\dashrightarrow\Sigma_g$ at $p_1'$ is 
a birational morphism $X\to\Sigma_g$ which factors through $\mu_1'$. 
Here the $(-1)$-curve contracted by the birational morphism 
$X_1\to\Sigma_g$ is the strict transform to $X_1$ of the fibre of $\Sigma_{g+1}$ passing through $p_1'$. 
In particular, it intersects with $(\mu'_1)_*F$ at just one point. 
Therefore, there exists a birational morphism $\mu:X\to\Sigma_g$ satisfying 
conditions $(\mathrm{i})$, $(\mathrm{ii})$ in Theorem~$\ref{Thm:PM}$ from the assumption. 

A base-point-free pencil $|F|$ can be considered as a subpencil 
$\Lambda\subset|2\Delta_{[g]}+(2g+1)\Gamma_{[g]}|$ with $(4g+4)$ simple base points on $\Sigma_{g}$ through $\mu$. 
If a base point of $\Lambda$ lies on $\Sigma_g\setminus\Delta_{[g]}$, then we obtain a birational morphism 
$X\to\Sigma_{g-1}$ satisfying conditions $(\mathrm{i})$, $(\mathrm{ii})$ in Theorem~$\ref{Thm:PM}$ 
from the elementary transformation $\Sigma_{g}\dashrightarrow\Sigma_{g-1}$ at the base point 
in the same way as above. 
It contradicts to the assumption. 
Thus a base point $p_1$ of $\Lambda$ is on $\Delta_{[g]}$. 
Since $\Delta_{[g]}.(2\Delta_{[g]}+(2g+1)\Gamma_{[g]})=1$, other base points $p_2, p_3, \ldots, p_{4g+4}$ of $\Lambda$ 
do not lie on $\Delta_{[g]}\setminus p_1$. 
Let $\Gamma_{[g],1}$ be the fibre of $\Sigma_g$ passing through $p_1$. 
Then the pull-back of $\Gamma_{[g],1}$ to $X$ is 
a unique degenerate fibre of the ruling $\Phi_{|\Gamma_{[g]}|}:X\to\mathbb{P}^1$. 
The dual graph is of type $A_{k+2}$ in Figure~$\ref{fig:DFofRulingA}$ with $k=4g+3$ or 
of type $D_{k+1}$ in Figure~$\ref{fig:DFofRulingD}$ with $k=4g+4$. 
Hence, we may assume that $p_{i+1}$ is infinitely near point of $p_i$ for simplicity. 

We support that the dual graph of the unique degenerate fibre $\mu^*\Gamma_{[g],1}$ of $\Phi_{|\Gamma_{[g]}|}$ 
is of type $A_{k+2}$ with $k=4g+3$. 
Then the tangential direction of $\Gamma_{[g],1}$ at $p_1$ is different from the direction corresponding to $p_2$. 
Let $X_3\to\Sigma_g$ be the blowing up at $p_1$, $p_2$ and $p_3$. 
We denote $(-1)$-curves corresponding to $p_1$, $p_2$ and $p_3$ by $E_1$, $E_2$ and $E_3$ respectively. 
Let $\mu_3:X\to X_3$ be the birational morphism which contracts $(4g+1)$ $(-1)$-curves except $E_1$, $E_2$ and $E_3$ 
among $(4g+4)$ $(-1)$-curves contracted by $\mu$. 
We consider the birational morphism $\tau:X_3\to\Sigma_d$ contracting 
$(\mu_3)_*(\mu^*\Gamma_{[g],1}-E_1)$, $(\mu_3)_*(E_1-E_2)$ and $(\mu_3)_*(E_2-E_3)$ in turn. 
We have that $(\tau\circ\mu_3)_*(\Delta_{[g]}-E_1)$ is a minimal section of $\Sigma_d$ and $d=g-1$ 
since $(\tau\circ\mu_3)_*(\Delta_{[g]}-E_1).(\tau\circ\mu_3)_*\Gamma_{[g]}=1$ and 
$(\tau\circ\mu_3)_*(\Delta_{[g]}-E_1)^2=-g+1\leq0$. 
In fact, $(\mu^*\Gamma_{[g],1}-E_1)$ is a $(-1)$-section of $f$. 
Furthermore, $\tau\circ\mu_3:X\to\Sigma_{g-1}$ is 
a birational morphism satisfying conditions $(\mathrm{i})$, $(\mathrm{ii})$ in Theorem~$\ref{Thm:PM}$. 
It contradicts to the assumption. 
Thus, the dual graph of the unique degenerate fibre $\mu^*\Gamma_{[g],1}$ of $\Phi_{|\Gamma_{[g]}|}$ 
is of type $D_{k+1}$ with $k=4g+4$. 
In particular, $p_2$ corresponds to the tangential direction of $\Gamma_{[g],1}$ at $p_1$. 
Hence, $f$ has a reducible fibre whose irreducible components are 
$(\Delta_{[g]}-E_1)$, $(\mu^*\Gamma_{[g],1}-E_1-E_2)$ and $(E_i-E_{i+1})$, $i=1, 2, \ldots, 4g+3$. 
The dual graph is as in Figure~$\ref{fig:VIII-4}$. 
\end{proof}

\begin{lemma}\label{Lem:2}
$(2\mathrm{a})$ and $(2\mathrm{b})$ in {\upshape Theorem~\ref{MT}} are equivalent. 
Furthermore, a reducible fibre of $f$ is unique. 
\end{lemma}
\begin{proof}
If we assume $(2\mathrm{a})$ in Theorem~\ref{MT}, then the equation of Mordell-Weil rank $(\ref{Eq:MWrk})$ 
implies the uniqueness of a reducible fibre of $f$. 
It is clear that $(2\mathrm{a})$ implies $(2\mathrm{b})$. 
In what follows, we assume $(2\mathrm{b})$. 
The case $g=1$ follows from the classification of singular fibres by Kodaira \cite{Kodai63}. 
We show the case $g\geq2$ by applying Corollary~\ref{Cor:DFofRuling} as follows. 

Let $F_0$ be a reducible fibre of $f$ whose dual graph contains the extended Dynkin diagram of 
unimodular integral lattice $D_{4g+4}^+$ in Figure~$\ref{fig:D_4g+4^+}$. 
We denote the irreducible component of $F_0$ whose self-intersection number is $(-g-1)$ by $\Theta_{4g+4}$. 
We remark that the number of irreducible components of $F_0$ is at most $(4g+5)$ 
from the equation of Mordell-Weil rank $(\ref{Eq:MWrk})$. 

At first, we support that the number of irreducible components of $F_0$ is exactly $(4g+4)$. 
From Corollary~\ref{Cor:DFofRuling}, there exists a unique $(-1)$-section $E_{4g+4}$ of $f$ 
which does not intersect with $\Theta_{4g+4}$. 
Let $\Gamma_{[d],0}$ be the degenerate fibre of $\Phi_{|(K_X+F)/(g-1)|}$ which 
consists of $E_{4g+4}$ and components of $F_0$ except $\Theta_{4g+4}$. 
We remark that the dual graph of $\Gamma_{[d],0}$ is of type $D_{k+1}$ with $k=4g+3$ in Figure~$\ref{fig:DFofRulingD}$. 
From Theorem~\ref{Thm:PM}, the ruling $\Phi_{|(K_X+F)/(g-1)|}$ has exactly one other degenerate fibre 
$\Gamma_{[d],\infty}$ whose dual graph is of type $A_{k+2}$ with $k=0$ in Figure~$\ref{fig:DFofRulingA}$. 
We remark that $\Gamma_{[d],\infty}$ consists of two $(-1)$-sections of $f$ which intersect with $F_0$ on $\Theta_{4g+4}$. 
We have a birational morphism $\mu:X\to\Sigma_d$ which satisfies conditions $(\mathrm{i})$, $(\mathrm{ii})$ 
in Theorem~$\ref{Thm:PM}$ by contracting $E_{4g+4}$ and $(4g+3)$ of components of $\Gamma_{[d],0}$ and $\Gamma_{[d],\infty}$. 
Then $\mu_*\Theta_{4g+4}$ intersects with the fibre $\mu_*\Gamma_{[d],0}$ of $\Sigma_d$ at one point. 
On the other hand, it intersects with the other fibre $\mu_*\Gamma_{[d],\infty}$ of $\Sigma_d$ at two point, 
which is absurd. 
Therefore, $F_0$ has exactly one other component $\Theta$. 

Next, we support $\Theta^2<-2$. 
In the quite same argument as the case where the number of irreducible components of $F_0$ is exactly $(4g+4)$, 
the ruling $\Phi_{|(K_X+F)/(g-1)|}$ has exactly two degenerate fibres $\Gamma_{[d],0}$ and $\Gamma_{[d],\infty}$ as above, 
and we have a birational morphism $\mu:X\to\Sigma_d$. 
Then $\mu_*\Theta_{4g+4}$ is a minimal section of $\Sigma_d$ since $\mu_*\Theta_{4g+4}.\mu_*\Gamma_{[d],0}=1$ and 
$(\mu_*\Theta_{4g+4})^2\leq (-g-1)+2=-g+1\leq-1$. 
From Corollary~\ref{Cor:DFofRuling}, $\Theta$ is not a component of a degenerate fibre of $\Phi_{|(K_X+F)/(g-1)|}$. 
Therefore, we have that $\mu_*\Theta$ is a section of $\Sigma_d$ from $(\mathrm{i})$ in Theorem~\ref{Thm:PM}. 
In particular, $(\mu_*\Theta)^2\geq d=-(\mu_*\Theta_{4g+4})^2\geq g-1$ and 
$\mu_*\Theta$ intersects with $\mu_*\Gamma_{[d],0}$ and $\mu_*\Gamma_{[d],\infty}$ transversally. 
Remark that $\Gamma_{[d],0}$ is a degenerate fibre of $\Phi_{|(K_X+F)/(g-1)|}$ of type $D_{k+1}$ with $k=4g+3$. 
Then we have $\Theta^2\geq (g-1)-2=g-3\geq-1$, which is absurd. 
Hence, we get $\Theta^2=-2$. 
Thus, we see that the dual graph of $F_0$ is as in Figure~$\ref{fig:VIII-4}$ from Corollary~\ref{Cor:DFofRuling}. 
\end{proof}

In the last, we prove the main theorem: 

\begin{proof}[\bf \boldmath Proof of Theorem~$\ref{MT}$]
We have that $(1)$ and $(2\mathrm{a})$ in Theorem~$\ref{MT}$ are equivalent from 
Lemma~$\ref{Lem:1}$ and $(\mathrm{II})$ in Lemma~$\ref{Lem:2b}$. 
It follows from $(\mathrm{IV})$ in Lemma~$\ref{Lem:2b}$ that $(2\mathrm{a})$ in Theorem~$\ref{MT}$ implies $(3)$. 
If we assume $(3)$ in Theorem~$\ref{MT}$, then the assumption in Lemma~$\ref{Lem:3}$ is satisfied, 
which leads to $(2\mathrm{a})$ in Theorem~$\ref{MT}$. 
Therefore, we see that $(1)$, $(2\mathrm{a})$, $(2\mathrm{b})$ and $(3)$ in Theorem~$\ref{MT}$ are 
equivalent from Lemma~$\ref{Lem:2}$. 
We show that $(5\mathrm{a})$ implies $(2\mathrm{a})$ from a standard calculation for a double cover. 
$(\mathrm{I})$ in Lemma~$\ref{Lem:2b}$ says that $(4\mathrm{a})$ in Theorem~$\ref{MT}$ follows from $(2\mathrm{a})$. 

Now, by taking $t$ and $y$ as coordinates of the first component and the second one of 
$\varSigma_0=\mathbb{P}^1\times\mathbb{P}^1$ respectively, 
$\psi(t,y,0)=0$ in $(5\mathrm{b})$ in Theorem~$\ref{MT}$ defines a branch divisor of a double cover 
$\varsigma\circ\pi^\flat:X\to\varSigma_0$ in $(5\mathrm{a})$. 
Furthermore, $\mathbb{C}[t,y,z]/(\psi(t,y,z))$ is an affine coordinate ring of the surface obtained from 
the finite double cover of $\varSigma_0$ whose branch divisor is defined by $\psi(t,y,0)=0$. 
Therefore, $(5\mathrm{a})$ follows from $(5\mathrm{b})$. 
Conversely, we support $(5\mathrm{a})$ firstly. 
We may assume that an intersection point of $\varDelta_{[0],0}$ and $\varGamma_{[0],0}$ is 
$(t,y)=(0,0)$ by performing projective transformations on $\varSigma_0$ if necessary. 
Let $\psi_B(t,y)=0$ be a defining equation of a branch divisor 
$(\varDelta_{[0],0}+\varGamma_{[0],0}+B_{[0],0})$ in $(5\mathrm{a})$. 
Put $\psi(t,y,z)=z^2-\psi_B(t,y)$. 
Then we have $(5\mathrm{b})$ from a standard calculation. 
Similarly, an equation in $(\mathrm{II})$ in $(4\mathrm{b})$ is exactly a defining equation of a pencil $\Lambda$ in 
$(4\mathrm{a})$ by performing a projective transformation of a base curve $\mathbb{P}^1$ of $f$ if necessary. 
In the result, $(4\mathrm{a})$ and $(4\mathrm{b})$ are equivalent.

In the last, we show that $(4\mathrm{a})$ and $(4\mathrm{b})$ imply $(5\mathrm{a})$ as follows. 
We consider on Zariski open defined by $X_0Y_0\not=0$ in $\Sigma_g$. 
Then the pencil $\Lambda$ as in $(4\mathrm{a})$ is defined by 
\begin{align}
\{ x^2y^{2g+1}=t(c_{2,2g+1}x^2y^{2g+1}+c_{2,2g}x^2y^{2g}+ \cdots +c_{2,1}x^2y+c_{2,0}x^2&\label{DeqPonSg}\\
+c_{1,g+1}xy^{g+1}+c_{1,g}xy^g+ \cdots +c_{1,1}xy+c_{0,1}y&)\}_{t\in\mathbb{P}^1}, \notag
\end{align}
where $c_{i,j}\in\mathbb{C}$, $i=0,1,2$, $j=0,1,\ldots,ig+1$ with $c_{1,0}=c_{0,0}=0$ and $c_{2,0}c_{0,1}\not=0$. 
The meromorphic map defined by $(x,y)\mapsto(t,y)$ induces 
the generically finite double cover $f\times\Phi_{|\Gamma_{[g]}|}:X\to\varSigma_0$ 
by eliminating the base points of $\Lambda$. 
It is the double cover as in Corollary~$\ref{Cor:FDC}$ when $g\geq2$. 
Therefore, the branch divisor $B$ of $f\times\Phi_{|\Gamma_{[g]}|}$ contains the curves defined 
the discriminant of the equation $(\ref{DeqPonSg})$ for $x$. 
On the other hand, the discriminant 
\begin{align*}
ty
\left(
\begin{array}{rrl}
4c_{0,1}y^{2g+1}\hspace{-10pt}&-4c_{2,0}c_{0,1}\hspace{-8pt}&t\\
&-4c_{0,1}\hspace{-10pt}&t(c_{2,2g+1}y^{2g+1}+c_{2,2g}y^{2g}+ \cdots +c_{2,1}y)\\
&+\hspace{-10pt}&ty(c_{1,g+1}y^g+c_{1,g}y^{g-1}+ \cdots +c_{1,2}y+c_{1,1})^2
\end{array}
\right)
\end{align*}
defines $\varDelta_{[0],0}$, $\varGamma_{[0],0}$ and a section $B_{[0],0}$ of $\varSigma_0$ 
which is linearly equivalent to $((2g+1)\varDelta_{[0]}+\varGamma_{[0]})$. 
In fact, the discriminant is a defining equation of $B=(\varDelta_{[0],0}+\varGamma_{[0],0}+B_{[0],0})$ 
from Corollary~$\ref{Cor:FDC}$. 
Furthermore, we see that 
$B_{[0],0}$ has a contact of order $(2g+1)$ with $\varGamma_{[0],0}$ 
at the intersection point $(0,0)$ of $\varGamma_{[0],0}$ with $\varDelta_{[0],0}$ 
by considering $c_{2,0}c_{0,1}\not=0$ and the defining equation. 
\end{proof}

%
%

\bigskip
\address{
General Education, \\
Gifu National College of Technology, \\
Kamimakuwa 2236-2, Motosu, \\ 
Gifu 501-0495, Japan
}
{kit058shiny@ gifu-nct.ac.jp}
\end{document}